\documentclass[11pt]{article}
\usepackage[numbers,sort&compress]{natbib}
\usepackage{enumerate}
\usepackage{amscd}
\usepackage{amsmath}
\usepackage{latexsym}
\usepackage{amsfonts}
\usepackage{amssymb}
\usepackage{amsthm}
\usepackage{verbatim}
\usepackage{mathrsfs}
\usepackage{enumerate}
\usepackage{hyperref}

\theoremstyle{plain}\newtheorem{definition}{Definition}[section]
\theoremstyle{definition}\newtheorem{theorem}{Theorem}[section]
\theoremstyle{plain}\newtheorem{lemma}[theorem]{Lemma}
\theoremstyle{plain}
\theoremstyle{plain}\newtheorem{prop}[theorem]{Proposition}
\theoremstyle{remark}\newtheorem{remark}{Remark}[section]
\usepackage{xcolor}
\newcommand{\wblue}[1]{\textcolor{black}{#1}}
\newcommand{\wred}[1]{\textcolor{black}{#1}}

\newcommand{\Div}{\mathrm{div}\,}
\newcommand{\B}{\Big}

\newcommand{\be}{\begin{equation}}
\newcommand{\ee}{\end{equation}}
 \newcommand{\ba}{\begin{aligned}}
 \newcommand{\ea}{\end{aligned}}

\providecommand{\bysame}{\leavevmode\hbox to3em{\hrulefill}\thinspace}
  \newcommand{\f}{\frac}
    
  \newcommand{\ben}{\begin{enumerate}}
   \newcommand{\een}{\end{enumerate}}

\newcommand{\ti}{\nabla}

\newcommand{\Rmnum}[1]{\expandafter\@slowromancap\romannumeral #1@}

\allowdisplaybreaks

\numberwithin{equation}{section}
\begin{document}
\title{On the box-counting dimension   of   potential singular set for suitable weak solutions to the 3D Navier-Stokes equations}
\author{Yanqing Wang\footnote{ Department of Mathematics and Information Science, Zhengzhou University of Light Industry, Zhengzhou, Henan  450002,  P. R. China Email: wangyanqing20056@gmail.com}\; and  Gang Wu\footnote{School of Mathematical Sciences,  University of Chinese Academy of Sciences, Beijing 100049, P. R. China Email: wugangmaths@gmail.com}}
\date{}
\maketitle
\begin{abstract}
In this paper, we are concerned with
  the  upper  box-counting  dimension of   the set of possible singular points in space-time  of suitable weak solutions to the 3D Navier-Stokes equations. By  taking full advantage of
   the pressure  $\Pi$ in terms of
$\nabla \Pi$  in equations, we show that
  this  upper  box dimension    is at most $135/104(\approx1.30)$, which   improves the
 known  upper box-counting dimension $95/63(\approx1.51)$ in
Koh et al. \cite[ J. Differential Equations, 261: 3137--3148, 2016]{[KY6]}, $45/29(\approx1.55)$ in
Kukavica    et al. \cite[Nonlinearity 25: 2775-2783, 2012]{[KP]} and  $135/82(\approx1.65)$ in  Kukavica
\cite[Nonlinearity 22: 2889-2900, 2009]{[Kukavica]}.

 \end{abstract}
\section{Introduction}
\label{intro}
\setcounter{section}{1}\setcounter{equation}{0}
We consider  the following   incompressible Navier-Stokes system in dimension 3
\be\left\{\ba\label{NS}
&u_{t} -\Delta  u+ u\cdot\ti
u+\nabla \Pi=0, \\
&\Div u=0,\\
&u|_{t=0}=u_0,
\ea\right.\ee
 where $u $ stands for the flow  velocity field, the scalar function $\Pi=-\Delta^{-1}\partial_{i}\partial_{j}(u_{i}u_{j})$ represents the   pressure.   The
initial  velocity $u_0$ satisfies   $\text{div}\,u_0=0$.

In his \wblue{seminal} work \cite{[Leray]}, Leray constructed the global weak solutions to the tri-dimensional Navier-Stokes equations, namely, for a divergence--free vector field $u_0\in L^2(\mathbb{R}^{3})$, there exists a   weak solution \wblue{$(u,\,\Pi)$} such that $u\in  L^{\infty}(0,T;L^{2}(\mathbb{R}^{3}))\cap L^{2}(0,T;H^{1}(\mathbb{R}^{3})) $ to the system \eqref{NS}.  It follows from the   interpolation inequality that
$u\in L^{p}(0,T;L^{q}(\mathbb{R}^{3})) ~~  \text{with}~\f{2}{p}+\f{3}{q}=\f{3}{2} ,~~2\leq p\leq \infty.$
 Solonnikov  \cite{[Solonnikov]} and Giga and Sohr \cite{[GS]} proved the following regularity of the finite energy weak solutions
 \be\ba   &\nabla\Pi\in L_{t}^{s}L_{x}^{q}  ~~~\text{with} ~~~\f{2}{s}+\f{3}{q}=4, ~ ~~~  \Pi\in L_{t}^{s}L_{x}^{j}  ~~~\text{with} ~~~\f{2}{s}+\f{3}{j}=3,\ \\ &u\in L^{m}_{t}L^{n}_{x} ~~~~\text{with} ~~~ \f{2}{m}+\f{3}{n}=2,~~m\geq s,~~ n>q,~1<s,q<\infty.\label{GS}\ea\ee
A priori estimates of solutions to the Navier-Stokes system in $\mathbb{R}^{3}$ such as \wred{$u \in L^{1}_{t}L^{\infty}_{x} $, $(-\Delta)^{\f{\alpha}{2}}\nabla^{d}u\in L_{t}^{4/(d+\alpha+1) ,\infty}L_{x}^{4/(d+\alpha+1) ,\infty}$ ($\alpha\in[0,2),d\geq1$), $\nabla^{s}u\in L_{t}^{\f{2}{2s-1}}L^{2}_{x}$} can be found in \cite{[FGT],[CV],[Vasseur1],[Constantin]} and   references
therein.

However,
the  regularity of weak solutions mentioned above can not lead to the full
  regularity of weak solutions. The   partial regularity result of weak solutions is also originated from Leray in
  \cite{[Leray]}, where it \wblue{was} shown that  one dimension Lebesgue measure of the set of the potential time singular
  points   for the weak solutions to 3D Navier-Stokes equations is zero.
This result was improved by Scheffer in \cite{[Scheffer]}, where he  showed that the set of possible time singular points has 1/2-dimensional measure zero.
In 1970s,  Scheffer  \cite{[Scheffer1],[Scheffer2],[Scheffer3]} also
 considered  the potential space-time singular points  set of solutions to the   Navier-Stokes equations via introducing the suitable weak solutions and proved that the Hausdorff dimension of the space-times singular points set of   suitable weak solutions
of the 3D Navier-Stokes equations  is at most $5/3$. The so-called suitable weak solutions is a kind of weak solutions  meeting   the local energy inequality.  A point is said to be a regular point  to the Navier-Stokes equations \eqref{NS} if one
has the $L^{\infty} $ bound of $u$ in some neighborhood
of this point.  The remaining points are called  singular point and denoted by $\mathcal{S}$.
 In this direction, the celebrated
  work  is that   one
dimensional Hausdorff measure of the possible space-time singular points set of suitable
weak solutions to the 3D Navier-Stokes equations is zero, which \wblue{was} proved by
Caffarelli,  Kohn,   Nirenberg in \cite{[CKN]}.
The Caffarelli-Kohn-Nirenberg theorem completely
rests  on   the following $\varepsilon$-regularity criterion:
 there is an absolute constant \wblue{$\varepsilon$} such that, if
   \be\label{ckn}
  \limsup_{r\rightarrow0}\f{1}{r}\iint_{Q(r)}|\nabla u|^{2}dxdt
  \leq \wblue{\varepsilon},
  \ee
  then $(0,\,0)$ is a regular point, where
   $Q(r):=B(r)\times(-r^{2},0)$ and $B(r)$
   denotes the ball of center $0$ and radius $r$.
Alternative approaches to  Caffarelli-Kohn-Nirenberg theorem
 and its generalization were presented in    several works (see, e.g., \cite{[TY],[RWW],[Lin],[LS],[CL],
[CY1],[Kukavica],[KP],[RWW1],[Vasseur],[WZ],[TX],[WW],[RS3]}).

Notice that
Hausdorff dimension is a  kind of  fractal dimension. Another widely used fractal  dimension
is \wblue{box-counting} or box dimension (Minkowski dimension).
As pointed out by   Falconer in \cite{[Falconer]},
box dimension's popularity is largely due to its relative ease of mathematical calculation and
empirical estimation, the definition goes back at least to the 1930s and it has been
variously termed Kolmogorov entropy, entropy dimension, capacity dimension, metric dimension,
logarithmic density and information dimension.
 Extensive study on the  box dimension   can be found in
 \cite{[Falconer]}.
 The definition of box dimension is via \wblue{lower} box
dimension and upper box
dimension.
It is valid that Hausdorff dimension \wblue{is less than upper box dimension (see e.g. \cite{[Falconer]})}.
For convenience, in what follows,   box
dimension means the upper box
dimension.
The  fractal upper box
dimension  of potential time  singular points set of finite energy  weak solutions to the 3D Navier-Stokes system is at most $1/2$, which was proved by
Robinson  and   Sadowski in \cite{[RS1]}.
By means of  $\varepsilon$-regularity criterion proved  in  \cite{[Lin],[LS]},
Robinson  and   Sadowski \cite{[RS2]}
obtained that the  fractal upper box
dimension   of possible space-time singular points set  of suitable weak solutions to the  three-dimensional Navier-Stokes equations   is at most 5/3.
It \wblue{was} shown that the parabolic fractal dimension of the
singular points set is less than or equal to $135/82(\approx1.65)$ by Kukavica in  \cite{[Kukavica]}. Later,
  Kukavica and Pei \cite{[KP]} improved this
 upper box
dimension to $45/29(\approx1.55)$.
Very recently,   Koh and  Yang  \cite{[KY6]} proved that
the  fractal upper box
dimension  of the set of potential  singular points   of suitable weak solutions in the Navier-Stokes equations    is bounded by $95/63(\approx1.51)$.
 The purpose of this paper is to refine the fractal upper box
 dimension   of  potential  singular points set to system \eqref{NS}.
Before we state the main theorems of this paper, we present the definition of upper box-counting dimension.
\begin{definition}\label{defibox}
The (upper) box-counting dimension of a set $X$ is usually defined as
$$d_{\text{box}}(X)=\limsup_{\epsilon\rightarrow0}\f{\log N(X,\,\epsilon)}{-\log\epsilon},$$
where $N(X,\,\epsilon)$ is the minimum number of balls of radius $\epsilon$ required to cover $X$.
\end{definition}
\begin{theorem}\label{the1.1}
 The (upper) box dimension of the set of possible  singular points  in space-time of suitable weak solutions to the 3D Navier-Stokes equations \eqref{NS} is at most $135/104(\approx1.30).$
\end{theorem}
\begin{remark}
This improves the previous box dimension of   potential singular set   of suitable weak solutions to the  tri-dimensional Navier-Stokes equations obtained in \cite{[KP],[RS2],[Kukavica],[KY6]}.
\end{remark}
As   calculating the Hausdorff dimension of singular points set of suitable weak solutions to the Navier-Stokes equations,
the  size   of  	potential  singular points set  in terms of box-counting measure strongly depends on the
$\varepsilon$-regularity criterion. To show Theorem \ref{the1.1}, we prove the following result.
 \begin{theorem}\label{the1.2}
Suppose that the pair $(u, \,\Pi)$ is a suitable weak solution to (\ref{NS}).  Then $|u|$   can be bounded by $1$ in some neighborhood of $(x_{0},\,t_{0})$
 provided the following condition holds,
\be\label{cond}\ba  \wblue{\iint_{\tilde{Q} (r)}
|\nabla u |^{2} +| u |^{ 10/3}+|\Pi-\overline{\Pi}_{\tilde{B}(r)} |^{ 5/3}+
 |\nabla \Pi| ^{5/4}dxdt \leq    r^{135/104}\varepsilon_{1},} \ea\ee
for an absolute constant $\varepsilon_{1}>0.$
\end{theorem}
\wblue{The notations used here can be found at the end of this section.}

\wblue{Next, we give} several remarks about this theorem.
\begin{remark}
Theorem \ref{the1.2} is an improvement of corresponding results
proved by Kukavica and  Pei  in \cite{[KP]}.
\end{remark}
\begin{remark}
The proof of \wblue{Theorem} \ref{the1.2} follows an approach utilized
in \cite{[KY6]}.  Compared with the proof of Koh and  Yang \cite{[KY6]},  first,
 notice that the appearance of the pressure in equations is in terms of
$\nabla \Pi$ rather than $\Pi$, thence, we
can always replace $\Pi$  by $\Pi-  \overline{\Pi  }_{\Omega}$. Based on this, we could
  make full use of  better decay estimate of pressure $\Pi-\overline{\Pi  }_{\Omega}$   than that of
\cite{[KY6]} (see Lemma \ref{presure} in Section \ref{sec2}). This is \wblue{partially motivated by the recent work in
\cite{[RWW]}, where} the authors used this trick
  to study the
partial regularity of suitable weak solutions to the
 Navier-Stokes equations with  fractional dissipation $(-\Delta)^{\alpha}$ in the case $\alpha=3/4$. Second,
  in the spirit of \cite{[KP]},
  we  will utilize the quantity $\|\nabla\Pi\|^{5/4}_{L^{5/4}_{t,x}}$  bounded by the
initial energy as widely as possible since the scaling of
  $\|\nabla\Pi\|^{5/4}_{L^{5/4}_{t,x}}$  is  better than that of $\|\Pi\|^{5/3}_{L^{5/3}_{t,x}}$.
This
enables us to obtain that box dimension    is at most $180/131(\approx1.37)$.
\end{remark}
\begin{remark}\label{r1.4}
Based on referees'  crucial  comments  	involving inequality \eqref{refer} and a very recent literature   \cite{[RWW1]}, we  realized that there  	exists an   appropriate  	 interpolation inequality   \eqref{inter2} for estimating the box dimension. In comparison with inequality   \eqref{inter1} of Koh and  Yang in \cite{[KY6]},
the inequality \eqref{inter2} allows us to take full advantage of  quantity  $ \|\nabla u\|^{2}_{L^{2}(Q(r))} $ in  \eqref{cond}. Though both $ \|u\|^{2}_{L^{\infty,2}(Q(r))}$ and $ \|\nabla u\|^{2}_{L^{2}(Q(r))} $ have the same scaling,  $ \|u\|^{2}_{L^{\infty,2}(Q(r))}$ appearing in \eqref{cond} seems to be inappropriate in discussion of  estimating the singular points set, that is,   $ \|\nabla u\|^{2}_{L^{2}(Q(r))} $  is  more useful  than $ \|u\|^{2}_{L^{\infty,2}(Q(r))}$ in our arguments. This further helps  us to improve the box dimension from  $180/131$ to   $135/104$.\end{remark}
\begin{remark}
To the knowledge of the authors, the strategy that one applies
the interior
estimate of harmonic function to the pressure $\Pi-  \overline{\Pi  }_{\Omega}$ is due to \cite[Lemma 2.1, p.222]{[Seregin]}.     Lemma \ref{presure} proved in \cite{[Seregin]} combined with the proof
 of \cite{[KY6]} yields that the upper box dimension
 is bounded by
  $635/441(\approx1.44).$
\end{remark}
\begin{remark}
A combination of the idea of \cite{[KY6]} and the quantity $\|\nabla\Pi\|^{5/4}_{L^{5/4}_{t,x}}$ as \cite{[KP]}  implies that,  the upper box dimension of singular points set is at most $75/51(\approx1.47)$.
\end{remark}
\begin{remark}
The upper bound on the Hausdorff dimension of the
potential space-time singular points of suitable weak solutions to the
generalized Navier-Stokes equations with $(-\Delta)^{\alpha}$ in the case  $3/4\leq\alpha<1$ is obtained in \cite{[TY],[RWW]}. An
interesting issue is
to estimate the
 box-dimension of singular points set of suitable weak solutions   of the fractional Navier-Stokes  equations  in the same case.
\end{remark}

With \wblue{Theorem \ref{the1.2} in hand, we can} present the proof of Theorem \ref{the1.1}.
\begin{proof}[Proof of Theorem \ref{the1.1}]
It follows from the definition of box-counting dimension that if $\delta<d_{\text{box}}(\mathcal{S})$, there exists a sequence $\epsilon_{j}\rightarrow0$ such that
$$N(\mathcal{S},\,\epsilon_{j})>\epsilon_{j}^{-\delta}.$$
We proceed by contradiction below. We assume that \wblue{$d_{\text{box}}(\mathcal{S})>135/104 $}, then we can    choose a constant $d$ such that \wblue{$180/131<d<d_{\text{box}}(\mathcal{S})$}.  Thus, there exists a decreasing sequence $\epsilon_{j}\rightarrow0$ such that
$$
N(\mathcal{S},\,\epsilon_{j})>\epsilon_{j}^{-d}.$$ Let ${(x_{i},t_{i})}^{N(\mathcal{S},\epsilon_{j})}_{i=1}$ be a collection of $\epsilon_{j}$-separated points in $\mathcal{S}$. According to Theorem \ref{the1.2}, for any $(x_{i},t_{i})\in \mathcal{S}$, we get
$$\ba
\wblue{\int_{t_{i}-\epsilon_{j}^{2}}^{t_{i}}
\int_{\tilde{B}_{i} ( \epsilon_{j})}
|\nabla u |^{2} +| u |^{ 10/3}+|\Pi-\overline{\Pi}_{\tilde{B}_{i}(\epsilon_{j})}|^{ 5/3}+
 |\nabla \Pi| ^{5/4}dxdt
> \epsilon_{j}^{135/104} \varepsilon_{1},} \ea$$
where $\tilde{B}_{i}(\mu):=B(x_{i}, \mu)$.\\
Combining the estimates above, we conclude   that
 $$\ba \sum^{N(\mathcal{S},\,\epsilon_{j})}_{i=1 }\int_{t_{i}-\epsilon_{j}^{2}}^{t_{i}}
\int_{\tilde{B}_{i} ( \epsilon_{j})}
|\nabla u |^{2} +| u |^{ 10/3}+| \Pi-\overline{\Pi}_{\tilde{B}_{i}(\epsilon_{j})} |^{ 5/3}+
 |\nabla \Pi| ^{5/4}dxdt
> \epsilon_{j}^{135/104-d } \varepsilon_{1}.  \ea$$
With the help of \eqref{GS}, we know  that  the left hand side of the above equality is bounded by   the initial data with finite-kinetic energy. Since \wblue{$d>135/104$, we get a contradiction as $j\rightarrow\infty$ and thus complete} the proof of Theorem \ref{the1.1}.
\end{proof}

The remainder of this paper unfolds as follows.
In Section  2, we recall the definition of suitable weak solutions to the Navier-Stokes equations.
  Then we  will list some
\wred{crucial} bounds for the scaling   invariant quantities. Section 3 contains the proof of Theorem \ref{the1.2}.

\noindent
{\bf Notations:} Throughout this paper, we denote
\begin{align*}
     &B(x,\mu):=\{y\in \mathbb{R}^{3}||x-y|\leq \mu\}, && B(\mu):= B(0,\mu), && \tilde{B}(\mu):=B(x_{0},\,\mu),\\
     &Q(x,t,\mu):=B(x,\,\mu)\times(t-\mu^{2}, t),  && Q(\mu):= Q(0,0,\mu), && \tilde{Q}(\mu):= Q(x_{0},t_{0},\mu).
\end{align*}
The classical Sobolev norm $\|\cdot\|_{H^{s}}$  is defined as   $\|f\|^{2} _{{H}^{s}}= \int_{\mathbb{R}^{n}} (1+|\xi|)^{2s}|\hat{f}(\xi)|^{2}d\xi$, $s\in \mathbb{R}$.
  We denote by  $ \dot{H}^{s}$ homogenous Sobolev spaces with the norm $\|f\|^{2} _{\dot{H}^{s}}= \int_{\mathbb{R}^{n}} |\xi|^{2s}|\hat{f}(\xi)|^{2}d\xi$.
 For $q\in [1,\,\infty]$, the notation $L^{q}(0,\,T;\,X)$ stands for the set of measurable functions on the interval $(0,\,T)$ with values in $X$ and $\|f(t,\cdot)\|_{X}$ belongs to $L^{q}(0,\,T)$.
   For simplicity,   we write
$$\wblue{\|f\| _{L^{q,\,\ell}(Q(\mu))}:=\|f\| _{L^{q}(-\mu^{2},\,0;\,L^{\ell}(B(\mu)))}~~~\text{ and}~~~~
  \|f\| _{L^{q}(Q(\mu))}:=\|f\| _{L^{q,\,q}(Q(\mu))}.} $$

 Denote
  the average of $f$ on the set \wblue{$\Omega$ by
  $\overline{f}_{\Omega}$}. For convenience,
  $\overline{f}_{r}$ represents  $\overline{f}_{B(r)}$.
  We will use the summation convention on repeated indices.
 $C$ is an absolute constant which may be different from line to line unless otherwise stated in this paper.

 \section{ Preliminaries}\label{sec2}
\setcounter{section}{2}\setcounter{equation}{0}

We start with the definition of the suitable weak solution to the Navier-Stokes equations \eqref{NS}.
\begin{definition}\label{defi}
A  pair  \wblue{$(u, \,\Pi)$} is called a suitable weak solution to the Navier-Stokes equations \eqref{NS} provided the following conditions are satisfied,
\begin{enumerate}[(1)]
\item $u \in L^{\infty}(-T,\,0;\,L^{2}(\mathbb{R}^{3}))\cap L^{2}(-T,\,0;\,\dot{H}^{1}(\mathbb{R}^{3})),\,\wblue{\Pi}\in
L^{3/2}(-T,\,0;L^{3/2}(\mathbb{R}^{3})).$\label{SWS1}
 \item$(u, ~\Pi)$~solves (\ref{NS}) in $\mathbb{R}^{3}\times (-T,\,0) $ in the sense of distributions;\label{SWS2}
 \item$(u, ~\Pi)$ satisfies the following inequality, for a.e. $t\in[-T,0]$,
 \begin{align}
 &\int_{\mathbb{R}^{3}} |u(x,t)|^{2} \phi(x,t) dx
 +2\int^{t}_{-T}\int_{\mathbb{R} ^{3 }}
  |\nabla u|^{2}\phi  dxds\nonumber\\ \leq&  \int^{t}_{-T }\int_{\mathbb{R}^{3}} |u|^{2}
 (\partial_{s}\phi+\Delta \phi)dxds
  + \int^{t}_{-T }
 \int_{\mathbb{R}^{3}}u\cdot\nabla\phi (|u|^{2} +2\Pi)dxds, \label{loc}
 \end{align}
 where non-negative function $\phi(x,s)\in C_{0}^{\infty}(\mathbb{R}^{3}\times (-T,0) )$.\label{SWS3}
\end{enumerate}
\end{definition}
In the light of the natural
scaling property of Navier-Stokes  equations,
we introduce the following dimensionless quantities:
\wred{\begin{align}&E(u,\,\mu)=\mu^{-1}\|u\|^{2}_{L^{\infty,2}(Q(\mu))},
&E&_{\ast}(\nabla u,\,\mu)=\mu^{-1}\|\nabla u\|^{2}_{L^{2}(Q(\mu))}
,  \nonumber\\
&E_{p}(u,\,\mu)=\mu^{p-5}\|u\|^{p}_{L^{p}(Q(\mu))}
,&P&_{5/4}(\nabla\Pi,\,\mu)= \mu^{-5/4}
\|\nabla\Pi\|^{5/4}_{L^{5/4}(Q(\mu))},\nonumber
\\
&P_{3/2}\Big(\Pi,\,\mu\Big)= \mu^{ -2}
\Big\|\Pi-\overline{\Pi}_{B(\mu)}\Big\|^{3/2}_{L^{3/2}(Q(\mu))},
&P&_{5/3}\Big(\Pi,\,\mu\Big)= \mu^{ -5/3}
\Big\|\Pi-\overline{\Pi}_{B(\mu)}\Big\|^{5/3}_{L^{5/3}(Q(\mu))}.
\nonumber
\end{align}}

We recall the following lemma involving interpolation inequality. It is worth pointing out that
the first part   slightly improves the
 corresponding  result obtained in \cite{[CKN],[Lin],[LS]}.
 As  stated  in Remark \ref{r1.4}, we will utilize the second one, which was usually used for
 improving the
Caffarelli-Kohn-Nirenberg theorem by a logarithmic factor
 in \cite{[RWW1],[CL],[CY1]}. To make our paper more self-contained and more readable, we outline the  proof of its general case \eqref{inter3}.
\begin{lemma}\label{lemma2.2}
For $0<\mu\leq\f{1}{2}\rho$ and $4\leq b\leq6$ ,~
there is an absolute constant $C$  independent of  $\mu$ and $\rho$,~ such that
\begin{align}
E_{3}(u,\, \mu) &\leq C \B(\f{\rho}{\mu}\B)^{3/2} E^{3/4}(u,\,\rho)E_{\ast}^{3/4}(\nabla  u,\,\rho)
+ C\B(\f{\mu}{\rho}\B)^3 {E^{3/2}(u,\,\rho)},\label{inter1}\\
E_{3}(u,\, \mu) &\leq C \B(\f{\rho}{\mu}\B)^{3/2} E^{1/2}(u,\,\rho) E_{\ast} (\nabla  u,\,\rho)
+ C\B(\f{\mu}{\rho}\B)^3 {E^{3/2}(u,\,\rho)}\label{inter2},\\
E_{3}(u,\,\mu) &\leq  C \left(\dfrac{\rho}{\mu}\right)^{ {3/2}}
E^{\f{b-3}{b-2}}(u,\,\rho)E^{\f{b}{2b-4}}_{\ast}( \nabla u,\,\rho)
    +C\left(\dfrac{\mu}{\rho}\right)
    ^{3}E^{3/2}(u,\,\rho)\label{inter3}.
\end{align}
\end{lemma}

\begin{proof}[Proof of  inequality \eqref{inter3}]
Taking advantage  of the triangle inequality, H\"older's inequality and the Poincar\'e-Sobolev  inequality, for any $3<b\leq6$, we know that
\begin{align}\nonumber
\int_{B(\mu)}|u|^{3}dx\leq& C\int_{B(\mu)}|u-\bar{u}_{{\rho}}|^{3}dx
+C\int_{B(\mu)}|\bar{u}_{{\rho}}|^{3} dx\\
\leq& C\B(\int_{B(\mu)}|u-\bar{u}_{{\rho}}|^{2}dx\B)
^{\f{b-3}{b-2}}
\B(\int_{B(\mu)}|u-\bar{u}_{{\rho}}|
^{b}dx\B)^{\f{1}{b-2}}+
\f{\mu^{3} C}{\rho^{\f{9}{2}}}\B( \int_{B(\rho)}|u|^{2}dx\B)^{3/2}\nonumber
\\
\leq& C\mu^{\f{6-b}{2b-4}}\B(\int_{B(\rho)}|u|^{2}dx\B)
^{\f{b-3}{b-2}}
\B(\int_{B(\mu)}|u-\bar{u}_{{\rho}}|
^{6}dx\B)^{\f{b}{6b-12}}+
  \f{\mu^{3}C}{\rho^{\f{9}{2}}}\B( \int_{B(\rho)}|u|^{2}dx\B)^{3/2}\nonumber
\\
\leq& C\mu^{\f{6-b}{2b-4}}\B(\int_{B(\rho)}|u|^{2}dx\B)
^{\f{b-3}{b-2}}
 \B(\int_{B(\rho)}|\nabla u|^{2}dx\B)^{\f{b}{2b-4}} + \f{\mu^{3}C}{\rho^{\f{9}{2}}}\B( \int_{B(\rho)}|u|^{2}dx\B)^{3/2}.\nonumber\label{lem2.31}
 \end{align}
Integrating in time on $(-\mu^{2},\,0)$ this inequality and using the H\"older inequality, for any $b\geq4$, we obtain
 \begin{align}
\iint_{Q(\mu)}|u|^{3}dxdt
\leq& C\mu^{\f{1}{2} }\B(\sup_{-\rho^{2}\leq t\leq0}\int_{B(\rho)}|u |^{2}dx\B)
^{\f{b-3}{b-2}}
 \B(\iint_{Q(\rho)}|\nabla u|^{2}dxdt\B)^{\f{b}{2b-4}}\nonumber\\&+
 C\f{\mu^{5}}{\rho^{\f{9}{2}}}\B(\sup_{-\rho^{2 }\leq t\leq0}\int_{B(\rho)}|u|^{2}dx\B)^{3/2},\nonumber
 \end{align}
which leads to   the desired inequality
$$
E_{3}(u,\,\mu) \leq  C \left(\dfrac{\rho}{\mu}\right)^{ {3/2}}
E^{\f{b-3}{b-2}}(u,\,\rho)E^{\f{b}{2b-4}}_{\ast}( \nabla u,\,\rho)
    +C\left(\dfrac{\mu}{\rho}\right)
    ^{3}E^{3/2}(u,\,\rho).
$$
\end{proof}

  One can  also establish the following   decay estimate of pressure
via the interior estimate
of harmonic function; see also \cite{[TX],[WZ],[CL],[LS],[Lin],[CKN],[WW]} for different versions. For its proof,    we refer the reader to \cite[Lemma 2.1, p.222]{[Seregin]}.
Thanks to
the pressure  $\Pi$ in terms of
$\nabla \Pi$  in equations,
 we can  invoke this lemma in the proof of Theorem \ref{the1.2}.
\begin{lemma}\label{presure}
For \wblue{$0<\mu\leq\f{1}{2}\rho$}, there is an absolute constant $C$ independent of $\mu$ and $\rho$ such that
\wred{\begin{equation}\label{beterrp}
\begin{split}
P_{3/2}(\Pi ,\mu)
&\leq C\left(\f{\rho}{\mu}\right)^{2}E_{3}(u, \rho)+
C\left(\f{\mu}{\rho}\right)^{5/2}P_{3/2}(\Pi, \rho).
\end{split}
\end{equation}}
\end{lemma}
Note that
if the pair $(u,\,\Pi)$ is a suitable weak solution, so is $(u,\,\Pi-\overline{\Pi})$, therefore, the following $\varepsilon$-regularity criterion proved in \cite{[LS],[Lin]} is valid.
\begin{prop}\label{Lin}
Let $(u,\,\Pi)$ be a suitable weak solution to (\ref{NS}) in $Q(1) $. There exists
$\varepsilon_0>0$ such that if
\be
\f{1}{r^{2}}\iint_{Q(r)}|u|^3+|\Pi-\wblue{\overline{\Pi}_{r}}|^{3/2}dxdt\leq \varepsilon_0,
\ee
then $u$ is regular in $Q(r/2)$.
\end{prop}

\section{Proof  of Theorem \ref{the1.2}}
\label{sec3}
\setcounter{section}{3}\setcounter{equation}{0}
In this section, following the pathway of \cite{[KY6]} together with the auxiliary lemmas  in \wblue{Section} \ref{sec2}, we present the proof of Theorem 1.2.

\begin{proof}[Proof of Theorem \ref{the1.2}]
Without loss of generality, we 	
assume that $(x_{0},\,t_{0})=(0,\,0)$.
As \cite {[KY6]}, we present the assumption below
\be\label{assume}
 \iint_{Q (2\rho)}
|\nabla u |^{2} +| u |^{ 10/3}+|\Pi-\overline{\Pi}_{2\rho} |^{ 5/3}+
 |\nabla \Pi| ^{5/4}dxdt \leq    (2\rho)^{ 5/3-\gamma}\varepsilon_{1},
\ee
hence, it suffices to prove that $\gamma<115/312$ and $\gamma$ will be  sufficiently  close to $115/312$.\\
First, we assert that  $E(u,\rho)\leq C\varepsilon_{1}^{3/5}\rho^{-\f{3\gamma}{5}}$.
 Indeed,
letting $\phi(x,t)$ be  a smooth positive function supported in $Q(2\rho)$ and with value $1$ on the ball $Q(\rho)$, then,
 employing the divergence free condition,
H\"older's inequality thrice and the Gagliardo-Nirenberg inequality, we derive that
$$\ba
\iint_{\wblue{Q(2\rho)}}(|u|^{2}- \overline{|u|^{2}}_{2\rho} )u\nabla \phi dxdt
&\leq C\rho^{-1}\Big\| |u|^{2}- \overline{|u|^{2}}_{2\rho}  \Big\|_{L^{10/7,15/8}(Q(2\rho))}
 \| u \| _{L^{10/3,15/7}(Q(2\rho))}\\
&\leq C\rho^{-1}\| u ^{2}\|_{L^{5/3}(Q(2\rho))}^{1/2}
  \|u\nabla u  \|^{1/2}_{L^{5/4}(Q(2\rho))}
  \| u \| _{L^{10/3,15/7}(Q(2\rho))}\\
&\leq C\rho^{-1/2}\| u  \|^{5/2}_{L^{10/3}(Q(2\rho))}
     \|\nabla u  \|^{1/2}_{L^{2}(Q(2\rho))},
  \ea  $$
and
$$\ba
\iint_{\wblue{Q(2\rho)}}(&\Pi-\overline{\Pi}_{2\rho})u\nabla \phi dxdt\\
&\leq C\rho^{-1}\Big\|\Pi-\overline{\Pi}_{2\rho} \Big\|_{L^{10/7,15/8}(Q(2\rho))}
 \| u \| _{L^{10/3,15/7}(Q(2\rho))}\\
&\leq C\rho^{-1}\|\Pi-\overline{\Pi}_{2\rho}
\|_{L^{5/3}(Q(2\rho))}^{1/2}
  \|\Pi-\overline{\Pi}\|^{1/2}_{L^{5/4,15/7}(Q(2\rho))}
  \| u \| _{L^{10/3,15/7}(Q(2\rho))}\\
&\leq C\rho^{-1/2}\|\Pi-\overline{\Pi}_{2\rho}
\|_{L^{5/3}}^{1/2}
  \|\nabla\Pi\|^{1/2}_{L^{5/4}(Q(2\rho))}
  \| u \| _{L^{10/3}(Q(2\rho))}.
  \ea  $$

  These inequalities combined with the local energy inequality \eqref{loc} leads to
\begin{align}
\sup_{-\rho^{2}\leq t<0}\int_{B(\rho )}|u|^{2}&dx
+2C\iint_{Q (\rho)}
|\nabla u|^{2}dxdt
\nonumber\\
\leq& C
\Big(\iint_{Q(2\rho)}|u|^{10/3}dxdt\Big)^{3/5}
 +  C\rho^{-1/2} \| u  \|^{5/2}_{L^{10/3}(Q(2\rho))}
     \|\nabla u  \|^{1/2}_{L^{2}(Q(2\rho)) }\nonumber\\&+ C\rho^{-1/2}\| u  \|_{{L^{10/3}}(Q(2\rho))}
 \|\Pi-\overline{\Pi}_{2\rho}
\|_{L^{5/3}(Q(2\rho))}^{1/2}
  \|\nabla\Pi\|^{1/2}_{L^{5/4}(Q(2\rho))}
      \nonumber\\
\leq& C\varepsilon_{1}^{3/5}\rho^{1-\f{3\gamma}{5}},
\label{loc1.2} \end{align}
where  we have used \eqref{assume} and assumed that $\gamma\leq5/12.$ As a consequence,
\be\label{E}
E(u,\rho)\leq C\varepsilon_{1}^{3/5}\rho^{-\f{3\gamma}{5}}.
\ee

Second, iterating \eqref{beterrp} in Lemma \ref{presure}, we infer  that
\be\label{referee}
P_{3/2}(\Pi,\theta^{N}\mu)\leq C\sum^{N}_{\wblue{k=1}}\theta^{-2+\f{5(k-1)}{2}}
E_{3}(u,\theta^{N-k}\mu)+C\theta^{5N/2}
P_{3/2}(\Pi, \mu).
\ee
In view of  the Poincar\'e-Sobolev  inequality and H\"older's inequality, we see that
\begin{align}
\nonumber
 \|\Pi-\overline{\Pi}_{ \mu}\|^{3/2}_{L^{3/2}(Q( \mu))} &\leq
 \|\Pi-\overline{\Pi}_{ \mu}\|^{1/2}_{L^{5/4,15/7}(Q( \mu))} \|\Pi-\overline{\Pi}_{ \mu}\|_{L^{5/3,30/23}(Q( \mu))}\nonumber\\
 &\leq
 C\rho^{1/2}\|\nabla \Pi \|^{1/2}_{L^{5/4}(Q( \mu))} \|\Pi-\overline{\Pi}_{ \mu}\|_{L^{5/3}(Q( \mu))}.
 \label{pressureinti}\end{align}
Dividing both sides of \eqref{pressureinti} by $\mu^{2}$, we have
$$
P_{3/2}(\Pi, \mu)\leq CP^{2/5}_{5/4}(\nabla\Pi,\,\mu)
P^{3/5}_{5/3}(\Pi,\,\mu).
$$
By inserting the above inequality into \eqref{referee}, we know that
\be\label{refer}
P_{3/2}(\Pi,\theta^{N}\mu)\leq C\sum^{N}_{\wblue{k=1}}\theta^{-2+\f{5(k-1)}{2}}
E_{3}(u,\theta^{N-k}\mu)+C\theta^{5N/2}
P^{2/5}_{5/4}(\nabla\Pi,\,\mu)
P^{3/5}_{5/3}(\Pi,\,\mu).
\ee

To proceed further, we set $r=\rho^{\alpha}=\theta^{N}\mu$,~$ \theta=\rho^{\beta}$, ~$r_{N}=\mu =\theta^{-N}r=\rho^{\alpha-N\beta}$, where $\alpha,\beta$ is determined by $\gamma$.
 Their precise selection will be given
   in the end.
 Hence, we derive from \eqref{refer} that
$$\ba
P_{3/2}(\Pi,r)+E_{3}(u,r)
&\leq C\sum^{N}_{\wblue{k=1}}\theta^{-2+\f{5(k-1)}{2}}
E_{3}(u,\theta^{-k}r)+C\theta^{5N/2}P^{2/5}_{5/4}
(\nabla\Pi,\,r_{N})
P^{3/5}_{5/3}(\Pi,\,r_{N})
\\
&=
C\sum^{N}_{\wblue{k=1}}\theta^{-2+\f{5(k-1)}{2}}
E_{3}(u,r_{k})+C\theta^{5N/2}P^{2/5}_{5/4}
(\nabla\Pi,\,r_{N})
P^{3/5}_{5/3}(\Pi,\,r_{N}),
\ea$$
where we have used the fact that $
E_{3}(u,r)\leq C\theta^{-2}E_{3}(u,\theta^{-1}r).
$ Our aim below is to resort Proposition \ref{Lin} to complete the proof.
To this end, we suppose that \wblue{$r_{N}\leq \rho$}, then, we can adopt \eqref{inter2} in Lemma \ref{lemma2.2}, \eqref{E} and the hypothesis \eqref{assume} to obtain
$$\ba
E_{3}(u,\, r_{k})&\leq C \Big(\f{\rho}{ r_{k}}\Big)^{\wred{\f{3}{2}}}E^{1/2}(u,\,\rho)
 E_{\ast} ( \nabla u,\,\rho)+\Big(\f{ r_{k}}{\rho}\Big)^{3}E^{3/2}(u,\,\rho)\\
  &\leq C\varepsilon_{1}^{9/10} \Big( \rho^{\f{13}{6}-\f{3}{2}(\alpha-\wred{k}\beta)-\f{13\gamma}{10}}
  +\rho^{3\alpha-3-3\wred{k}\beta-\f{9\gamma}{10}} \Big).
 \ea$$
Therefore, for sufficiently large $N$, some elementary  calculations yield
$$\ba
\sum^{N}_{k=1}\theta^{-2+\f{5(k-1)}{2}}
E_{3}(u,r_{k})&\leq C\varepsilon_{1}^{9/10}\sum^{N}_{k=1}\B(\rho^{4 k\beta-\f{9\beta}{2}+\f{13}{6}-\f{3\alpha}{2} -\f{13\gamma}{10}}
  +\rho^{ -\f{ k\beta}{2}-\f{9\beta}{2}+    3\alpha-3 -\f{9\gamma}{10}}\B) \\
    &\leq C\varepsilon_{1}^{9/10}\B(\rho^{ -\f{\beta}{2}+\f{13}{6}-\f{3\alpha}{2} -\f{13\gamma}{10}}
  +\rho^{-\f{N\beta}{2}-\f{9\beta}{2}+ \wblue{3\alpha}-3-\f{9\gamma}{10}}\B)
  \\
    &\leq C\varepsilon_{1}^{9/10} \rho^{-\f{11\beta}{6}+\f{4}{9}
    -\f{7\gamma}{6}-\f{N\beta}{6}},
  \ea $$
where $\alpha$ is determined from
$$-\f{\beta}{2}+\f{13}{6}-\f{3\alpha}{2} -\f{13\gamma}{10}=
-\f{N\beta}{2}-\f{9\beta}{2}+    3\alpha-3-\f{9\gamma}{10}.$$
In short,
\be\label{aerfa}
\alpha=\f{2}{9}(4\beta+\f{31}{6}-\f{2\gamma}{5}+\f{1}{2}N\beta).
\ee

According to $r_{N}\leq \rho$    assumed above, namely,
 \be\label{c3}
\rho^{\alpha-N\beta}\leq \rho,\ee
\eqref{assume} and \eqref{aerfa}, we get
$$\ba
 \theta^{5N/2}P^{2/5}_{5/4}
(\nabla\Pi&,\,r_{N})
P^{3/5}_{5/3}(\Pi,\,r_{N})\\ \leq&  \rho^{\f{5N\beta}{2}} r_{N}^{-\f{3}{2}}
\B(\iint_{Q(2\rho)}|\nabla\Pi|^{5/4}dxdt\B)^{2/5}
\B( \iint_{Q(r_{N})}|\Pi-\overline{\Pi}_{2\rho}|^{5/3}dxdt\B)^{3/5}\\
\leq&  \rho^{\f{5N\beta}{2}}r_{N}^{-\f{3}{2}}
\B(\iint_{Q(2\rho)}|\nabla\Pi|^{5/4}dxdt\B)^{2/5}
\B(\iint_{Q(2\rho)}|\Pi-\overline{\Pi}_{2\rho}|^{5/3}dxdt\B)^{3/5}\\
\leq&  C\rho^{\f{23}{6} N\beta -\f{4\beta}{3}-\f{9\gamma}{10}}\varepsilon_{1}.
 \ea$$
 To conclude, by  Proposition \ref{Lin},
  we need  $-\f{11\beta}{6}+\f{4}{9}
    -\f{7\gamma}{6}-\f{N\beta}{6}\wblue{\geq 0}$ and $\f{23}{6} N\beta -\f{4\beta}{3}-\f{9\gamma}{10}\wblue{\geq 0}$, moreover, from \eqref{c3}, we know that  $\alpha-N\beta-1\geq0.$\\
In summary, the 	
index $\gamma$ should satisfy
$$
\gamma\leq \min\Big\{
\f{115N\beta-40\beta}{27},
\f{8-3N\beta-33\beta}{21},
\f{5-30N\beta+30\beta}{3},\f{5}{12}
\Big\},
$$
which means that $N\beta=9/104$ is appropriate.

Eventually, for any fixed  $\gamma<115/312$, we choose $N$
sufficiently large such that
$$
\beta=\f{9}{104N}\leq\f{7}{11}\B(\f{115}{312}-\gamma\B).
$$
Then, we select a fixed $\rho<1$ such that $\rho^{\beta}<1/2$.
With $\gamma,N,\beta$ in hand, we pick $\alpha=
\f{2}{9}(4\beta-\f{2\gamma}{5}+\f{3251}{624}).
$
This  allows us to get
$$
P_{3/2}(\Pi,r)+E_{3}(u,r)
 \leq C\varepsilon_{1}^{9/10}\leq \varepsilon_0,$$
with $r=\rho^{\alpha}$.
This completes the proof of Theorem \ref{the1.2}.
\end{proof}


\section*{Acknowledgement}
The authors would like to express their deepest gratitude to    two anonymous referees  and   the editors
   for  careful reading of our manuscript, the invaluable comments and suggestions which helped to improve the paper greatly,   especially, enlightening comments on  inequality \eqref{refer}.
   In addition,
  we would also like to express our thanks to  Prof. Minsuk Yang for pointing out that the proof of \eqref{inter1} in Lemma \ref{lemma2.2} can be found in \cite{[CY1]}.
Wang was partially supported by NSFC (No. 11601492).
Wu was partially supported by NSFC (No.11101405).


\begin{thebibliography}{00}


\bibitem{[CKN]}
L. Caffarelli,  R. Kohn and L. Nirenberg,  Partial regularity
of suitable weak solutions of Navier-Stokes equation, {\it Comm. Pure. Appl. Math., }   \textbf{35} (1982), 771--831.


\bibitem{[CL]}
H.  Choe, and J.  Lewis, On the singular set in the Navier-Stokes equations, {\it J. Funct. Anal.} \textbf{175}   (2000) 348-369.


\bibitem{[CY1]}
H.  Choe and M. Yang, Hausdorff measure of the singular set in the incompressible magnetohydrodynamic equations, {\it Comm. Math. Phys.} \textbf{336}   (2015) 171--198.


\bibitem{[CV]}
K. Choi and
A. Vasseur,  Estimates on fractional higher derivatives of
weak solutions for the Navier-Stokes equations.
{\it  Ann. Inst. H. Poincar\'e Anal. Non Lin\'eaire, } \textbf{31} (2014), 899--945.


 \bibitem{[Constantin]}
P. Constantin, Navier-Stokes equations and area of interfaces, {\it Comm. Math.
Phys.,} \textbf{129} (1990), 241--266.


\bibitem{[Falconer]}
K. Falconer, Fractal Geometry: Mathematical
Foundations and Applications (New York: Wiley) 1990.


 \bibitem{[FGT]}
 C. Foias, C. Guillop\'e and R. Temam, New a priori estimates for Navier-Stokes
equations in Dimension 3, {\it Comm. Partial Diff. Equat.,} \textbf{6} (1981),  329--359.

\bibitem{[GS]}
 Y. Giga and H. Sohr, Abstract $L^p$-estimates for the Cauchy problem with applications to the
Navier-Stokes equations in exterior domains,  {\it J. Funct. Anal., } \textbf{102}  (1991),  72--94.



\bibitem{[KY6]}
Y. Koh and M. Yang,
The Minkowski dimension of interior singular points in
the incompressible Navier-Stokes equations. {\it  J. Differential Equations.,}    \textbf{261} (2016),  3137--3148.

\bibitem{[Kukavica]}
I. Kukavica, The fractal dimension of the singular set for solutions of the Navier-Stokes system
{\it Nonlinearity.,} \textbf{22} (2009),  2889--2900.
\bibitem{[KP]}I.
Kukavica and Y. Pei, An estimate on the parabolic fractal dimension of the singular set for solutions of the Navier-Stokes system. {\it Nonlinearity.,} \textbf{25} (2012),  2775--2783.




































\bibitem{[LS]}O.  Ladyzenskaja and G.  Seregin,  On
partial regularity of suitable weak solutions to the three-dimensional Navier-Stokes equations, {\it   J. Math.
Fluid Mech.,}  \textbf{1}   (1999),  356--387.

\bibitem{[Leray]}  J. Leray, Sur le mouvement d\'eun liquide visqueux
emplissant l\'espace, {\it Acta Math.,} \textbf{63} (1934),  193--248.

\bibitem{[Lin]}F. Lin,   A new proof of the Caffarelli-Kohn-Nirenberg
Theorem,    {\it Comm. Pure Appl. Math.,}   \textbf{51} (1998),  241--257.




\bibitem{[RWW]} W. Ren, Y. Wang and G. Wu,
Partial regularity of suitable weak solutions to the multi-dimensional generalized magnetohydrodynamics equations. {\it Commun. Contemp. Math.,} (2016). Doi: 10.1142/S0219199716500188.




\bibitem{[RWW1]}
\bysame,
General logarithmic improvement on the  Caffarelli-Kohn-Nirenberg theorem, submitted for publication. 2016.



\bibitem{[RS1]}
J. Robinson  and W. Sadowski, Decay of weak solutions and the singular set of the three-dimensional
Navier-Stokes equations, {\it Nonlinearity.,} \textbf{20} (2007), 1185--1191.


\bibitem{[RS2]}
\bysame, Almost-everywhere uniqueness of Lagrangian trajectories for suitable
weak solutions of the three-dimensional Navier-Stokes equations, {\it Nonlinearity.,} \textbf{22}, (2009) 2093--2099.


\bibitem{[RS3]}
\bysame,
On the Dimension of the Singular Set of Solutions to the Navier-Stokes Equations, {\it Comm. Math.
Phys.,}
 \textbf{ 309}  (2012),  497--506.








\bibitem{[Scheffer]}
V. Scheffer,  Turbulence and Hausdorff dimension,
in Turbulence and the Navier-Stokes Equations,
 Lecture Notes in Math.,  Springer-Verlag,  \textbf{565} (1976),    94--112.
\bibitem{[Scheffer1]}
\bysame,   Partial regularity of solutions to the Navier-Stokes
equations,     {\it  Pacific J. Math.,}  \textbf{66} (1976),  535--552.

\bibitem{[Scheffer2]}\bysame,  Hausdorff measure and the Navier-Stokes equations,   {\it  Comm. Math. Phys.,} \textbf{55} (1977),  97--112.

\bibitem{[Scheffer3]}
\bysame,  The Navier-Stokes equations in space dimension
four,     {\it  Comm. Math. Phys.,}  \textbf{61} (1978),  41--68.


\bibitem{[Seregin]}
G. Seregin, On smoothness of $L_{3,\infty}$-solutions to the Navier-Stokes equations up to boundary. {\it Math. Ann., } \textbf{332} (2005),   219--238.



\bibitem{[Solonnikov]}
V. A. Solonnikov. Estimates for solutions of nonstationary Navier-Stokes equations. {\it J. Soviet Math.,} \textbf{8} (1977), 467--523.



\bibitem{[TY]}
 L. Tang and Y. Yu, Partial regularity of suitable weak solutions to the fractional Navier-Stokes equations. {\it Comm. Math. Phys., } \textbf{334} (2015), 1455--1482.


\bibitem{[TX]}
G. Tian and Z. Xin, Gradient estimation on Navier-Stokes equations, {\it Comm. Anal. Geom.,}  \textbf{7}  (1999),  221--257.

\bibitem{[Vasseur]}A. Vasseur,    A new proof of partial regularity of
solutions to Navier-Stokes equations, {\it    NoDEA Nonlinear Differential Equations Appl.,}    \textbf{14} (2007), 753--785.

\bibitem{[Vasseur1]}
\bysame. Higher derivatives estimate for the 3D Navier-Stokes equation. {\it
Ann. Inst. H. Poincar\'e Anal. Non Lin\'eaire, } \textbf{27} (2010),  1189--1204.

\bibitem{[WZ]}
W. Wang and Z. Zhang,   On the interior regularity criteria and the number of singular points to the Navier-Stokes equations.   {\it J. Anal. Math.,}  \textbf{123} (2014), 139--170.

 \bibitem{[WW]}Y. Wang and G. Wu,
A unified proof on the partial regularity for  suitable weak solutions of non-stationary and  stationary  Navier-Stokes equations.  {\it  J. Differential Equations.,}    \textbf{256} (2014),  1224--1249.














\end{thebibliography}
\end{document}